\documentclass[1p,12pt]{elsarticle}

\usepackage{amssymb,amsmath,amsthm}
\newtheorem{them}{Theorem}[section]
\newtheorem{coro}{Corollary}[section]
\newtheorem{lem}{Lemma}[section]
\newtheorem{pro}{Proposition}[section]

\begin{document}

\begin{frontmatter}




\title{Normal property, Jamenson property, CHIP and linear regularity for an infinite system of convex sets in Banach spaces\tnoteref{1}}
\tnotetext[1]{This research was supported by the National Natural Science
Foundation of P. R. China (Grant No. 11261067), the Scientific Research Foundation of Yunnan University under grant No. 2011YB29 and by IRTSTYN. }


\author[2]{Zhou Wei\corref{4}}
\ead{wzhou@ynu.edu.cn}
\author[2]{Qinghai He}
\ead{heqh@ynu.edu.cn}
\cortext[4]{ Corresponding author.}
\address[2]{Department of Mathematics, Yunnan University, Kunming 650091, P. R. China}


\begin{abstract}
In this paper, we study different kinds of normal
properties for infinite system of arbitrarily many convex sets in
a Banach space and provide the dual characterization for the normal
property in terms of the extended Jamenson property for
arbitrarily many weak$^*$-closed convex cones in the dual space.
Then, we use the normal property and the extended Jamenson
property to study CHIP, strong CHIP and linear regularity for the infinite
case of arbitrarily many convex sets and establish equivalent relationship
among these properties. In particular, we extend main results in [3] on
normal property, Jamenson property, CHIP and linear regularity for
finite system of convex sets in a Hilbert space to the
infinite case of arbitrarily many convex sets in Banach space
setting.
\end{abstract}

\begin{keyword}
normality\sep  CHIP\sep linear regularity


\MSC{41A65, 52A05, 90C25}
\end{keyword}

\end{frontmatter}

\section{Introduction}
The concepts of normal property, Jamenson property (property
$(G)$), conical hull intersection property (CHIP) and linear
regularity for finitely many closed convex sets are well-known in approximation theory and optimization and have played important roles in various branches of optimization such as convex feasibility problem,
constrained approximation, Fenchel duality, systems of convex
inequalities and error bounds (see [1, 2, 4, 7, 9, 11, 14, 15] and
references therein). The main aim of this paper is to study
normal property, property ($G$), CHIP and linear regularity for an
infinite system of arbitrarily many convex sets in a Banach space,
and establish equivalent interrelationship among
all these notions.

The notions of normal property and property ($G$) could be traced to
Jamenson in the early 1970s. Jamenson [20] introduced the concept of
normal property for two convex cones and used it to study the
closedness for the sum of dual cones. Property $(G)$ (called by
Jamenson) actually is the dual form of the normal property and was
used to establish a duality theory for two closed convex cones.
Afterwards, property $(G)$ has been extended to finite subsets and
played an important role in the study on CHIP and linear regularity
for a collection of finite closed convex sets. Bakan, Deutsch, and Li [3] extended the definition of normal
property for two closed convex cones to a finite collection of
convex sets in a Hilbert space and applied various normal properties
to study CHIP and linear regularity in Hilbert space setting.

The concept of CHIP was first introduced by Chui, Deutsch and Ward
in [10] for the problem of constrained best approximation in a
Hilbert space. They used the CHIP to provide a unifying framework
for the basic results in the subject of optimal constrained
approximation. Subsequently, strong CHIP, a stronger concept than
CHIP, was introduced by Deutsch, Li and Ward [15] and used
in constrained interpolation and constrained best approximation to
characterize a strong relationship for a certain pair of
optimization in Hilbert space setting. Strong CHIP actually turns out
to be a geometric version of the basic constraint qualification in
convex optimization (cf. [12]). Note that Jeyskumar [21] used strong
CHIP to study optimality and strong duality in convex programming and proved that the strong CHIP is equivalent to a complete Lagrange
multiplier characterization of optimality for convex programming
model problems.

In the 1990s, Bauschke and Borwein [5] studied the concept of
linear regularity when finding the projection from a point to an
intersection of finite many closed convex sets. They proved that this
notion played a key role in establishing a linear convergence rate
of iterates generated by the cyclic projection. Afterwards Bauschke,
Borwein and Li [7] showed that bounded linear regularity implies
strong CHIP for the case of finite closed convex sets but the
converse is not true necessarily (see [2, 8]). For the reason in
application of optimization and mathematical programming,
linear regularity has been extensively studied by many
authors. Readers could
refer to [5-9, 19, 22-24, 27-30] for results on necessary and/or
sufficient conditions for linear regularity .

It is an interesting and important topic to study relationship among
these notions aforementioned. Bauschke, Borwein and Li [7] exhibited
the relationship among strong CHIP, bounded linear regularity,
property $(G)$ and error bound. In 2005, Bakan, Deutsch and Li [3]
extended the normal property for two convex cones to a finite
collection of convex sets in a Hilbert space and used various normal
properties to study property $(G)$, CHIP, strong CHIP and linear
regularity. They provided a perspective on the relationship
among these notions. To our best knowledge, both works in [7] and
[3] are on the finite collection of convex sets in an Euclidean or a
Hilbert space while concepts of CHIP, strong CHIP and linear
regularity for the infinite case are extensively studied by many
authors, so it is meaningful and valuable to consider the normal
property and property $(G)$ for an infinite system of arbitrarily
many convex sets in Banach space setting, and then further study the
relationship among all these notions in the infinite case. Motivated by this and inspired by [3], in this paper, we mainly study
variations of normal properties, such as normal property,
weak normal property and uniform normal property, for a
collection of infinitely many convex sets in a Banach space and
establish the equivalence relationship between the normal property
and the extended Jamenson property (see definitions in section 3).
Then, we apply the normal property and  the extended Jamenson
property to study CHIP, strong CHIP and linear regularity for
the infinite system of arbitrarily many convex sets and provide equivalent relationships
among these notions(see
Theorems 5.4 and 5.5) so as to extend main results on these notions given in [3].

The paper is organized as follows. In section 2 we collect
definitions, notations and preliminary results used later in our
analysis. In section 3, we give definitions of various normal
properties for an infinite system of arbitrarily
many convex sets, and list some facts about different normal
properties. Section 4 is devoted to characterizations for the weak
normal property by the main tool of inverse sum and establishing the
equivalence between the normal property and its dual form in terms
of inverse sum and the extended Jamenson property for infinitely
many weak$^*$-closed convex cones in the dual space. Finally, we
draw main attentions to different kinds of CHIP and linear
regularity, and apply the main results established in sections 3 and
4 to characterize linear and bounded linear regularity in
Banach spaces.

\setcounter{equation}{0}
\section{Preliminaries} Let $X$ be a Banach space with the closed unit
balls denoted by $B_X$, and let $X^*$ denote the dual space of $X$.
For $\bar x\in X$ and $\delta>0$, let $B(\bar x, \delta)$ denote the
open ball with center $\bar x$ and radius $\delta$.

For a subset $S$ in $X$, let $\overline S$ be the closure of $S$ in
the norm topology. A subset $K$ in $X$ is said to be a {\it convex
cone}, if $K+K\subset K$ and $tK\subset K$ for all $t\geq 0$. The
{\it conical hull} of $S$, denoted by cone$(S)$, is the intersection
of all convex cones that contains $S$; thus cone$(S)$ is the smallest
convex cone that contains $S$. The {\it polar} of $S$ is the set
\begin{equation}
S^{\circ}:=\{x^*\in X^*: \langle x^*, x\rangle\leq 1\ \ \forall x\in
S\},
\end{equation}
and for a subset $P$ in $X^*$, the {\it polar}
of $P$ is the set
\begin{equation}
P^{\circ}:=\{x\in X: \langle y^*, x\rangle\leq 1\ \ \forall y^*\in
P\}.
\end{equation}

Let $A$ be a convex subset of $X$. The {\it dual cone} (or negative
polar) of $A$ is the set
\begin{equation}
A^{\ominus}:=\{x^*\in X^*: \langle x^*, x\rangle\leq 0\ \ \forall
x\in A\}.
\end{equation}
The {\it normal cone} and {\it tangent cone} to $A$ at $x\in A$ are
the sets $N(A, x):=(A-x)^{\ominus}$ and $T(A, x):=\overline{\rm
cone}(A-x)$, respectively. It is known that $N(A, x)$ and $T(A, x)$
are the dual cones (and polar) of each other; that is
\begin{equation}
N(A, x)^{\circ}=N(A, x)^{\ominus}=T(A, x)\ \ {\rm and}\ \ T(A,
x)^{\circ}=T(A, x)^{\ominus}=N(A, x).
\end{equation}
The {\it recession cone} of $A$ is the set
\begin{equation}
0^+A:=\{x\in X: A+tx\subset A\ \ \forall t\geq 0\}.
\end{equation}
Obviously, $0^+A$ is a convex cone, and $0^+A=A$ if and only if $A$
is a convex cone. It is not hard to verify that
\begin{equation}
0^+A^{\circ}=A^{\ominus}.
\end{equation}
In the case when $A$ is a closed convex set, $0^+A$ is a closed
convex cone and
\begin{equation}
0^+A=\{x\in X: y+\mathbb{R}_+x\subset A\  {\rm for}\ {\rm some}\
y\in A\}.
\end{equation}

The following lemma contains a few facts about polar and polar
operation. Readers are invited to see [13, 18, 25, 26] for details.

\begin{lem}
Let $\{A_i: i\in I\}$ be a collection of convex sets and $A$ be a
convex set in $X$.

(a) $A^{\circ}$ is a weak$^*$-closed convex subset of $X^*$ with $0\in
A^{\circ}$ and $A^{\circ}=\big(\overline A\big)^{\circ}$.

(b) If $A_1\subset A_2$, then $A_1^{\circ}\supset A_2^{\circ}$.

(c) $(\lambda A)^{\circ}=\frac{1}{\lambda}A^{\circ}$ for each
$\lambda>0$.

(d) (Bipolar theorem) $A^{\circ\circ}=\overline{\rm co}(A\cup
\{0\})$. In particular, if $0\in A$, then $A^{\circ\circ}=\overline
A$.

(e) If each $A_i$ is closed and $0\in \bigcap\limits_{i\in I}A_i$, then
\begin{equation}
\Big(\bigcap_{i\in I} A_i\Big)^{\circ}=\overline{\rm co}^{w^*}\Big(\bigcup_{i\in I}A_i^{\circ}\Big).
\end{equation}
\end{lem}

The next lemmas are about the closure, closedness and interior of sets
which will be used in our analysis later.

\begin{lem} Let $A_1$ and $A_2$ be subsets of $X$. Then
\begin{equation}
\overline A_1+\overline A_2\subset \overline{\overline A_1+\overline
A_2}=\overline{A_1+A_2}.
\end{equation}
\end{lem}

\begin{lem} Let $A$ be a convex subset of $X$ with ${\rm{int}}(A)\neq\emptyset$. Then
\begin{equation*}
{\rm{int}}(\overline A)={\rm{int}}(A).
\end{equation*}
\end{lem}

In the whole paper, we suppose that $I$ is an arbitrary nonempty
index set. Let $\{C_i: i\in I\}$ be a collection of subsets in $X$
with the nonempty intersection. The set $\sum\limits_{i\in I}C_i$ is
defined by
\begin{equation}
\sum_{i\in I}C_i:=\left\{\sum_{i\in I_0}c_i: c_i\in C_i,
\emptyset\not=I_0\subset I\ {\rm being}\ {\rm finite}\right\}.
\end{equation}

Let $P$ and $Q$ be metric spaces. Recall that a set-valued mapping
$F:P\rightarrow 2^Q$ is lower semicontinuous if, for any $x_0\in P$,
$y_0\in F(x_0)$ and any neighborhood $V$ of $y_0$, there exists a
neighborhood $U$ of $x_0$ such that $V\cap F(x)\not=\emptyset$ for
each $x\in U$. It is clear that $F:P\rightarrow 2^Q$ is lower
semicontinuous if and only if, for each $y\in Q$, the real-valued
function $x\mapsto d(y, F(x))$ is upper semicontinuous.

 When $P$ has finite elements, every set-valued
mapping $F:P\rightarrow 2^Q$ is automatically lower semicontinuous.

\setcounter{equation}{0}
\section{Normal properties} In a Hilbert space, Bakan, Deutsch and Li [3] introduced and studied various kinds
of normal properties such as uniform normal property, normal
property and weak normal property for a collection of finitely many convex sets. In this
section, we consider these normal properties for an infinite system of
arbitrarily many convex sets in a Banach space, and study the interrelationship among these normal properties.\\

\noindent{\bf{Definition 3.1.}} Let $\{A_i: i\in I\}$ be a collection of
convex sets in $X$ with a nonempty intersection.

(i) $\{A_i: i\in I\}$ is said to have the {\it closed intersection property}, if
\begin{equation}\label{3.1}
\overline{\bigcap_{i\in I}A_i}=\bigcap_{i\in I}\overline{A_i}.
\end{equation}

(ii) $\{A_i: i\in I\}$ is said to have the {\it normal property}, if there exists $\eta>0$ such that
\begin{equation}\label{3.2}
\bigcap_{i\in I}(A_i+\eta B_X)\subset\Big(\bigcap_{i\in I}A_i\Big)+B_X.
\end{equation}

(iii) $\{A_i: i\in I\}$ is said to have the {\it weak normal property}, if for every $x^*\in X^*$ there exists $\eta_{x^*}>0$ such that
\begin{equation}\label{3.3}
\bigcap_{i\in I}(A_i+\eta_{x^*} B_X)\subset\overline{\Big(\bigcap_{i\in I}A_i\Big)+\{x^*\}^{\circ}}.
\end{equation}

(iv) $\{A_i: i\in I\}$ is said to have the {\it uniform normal
property}, if there exists $\eta>0$ such that
\begin{equation}\label{3.4}
\bigcap_{i\in I}(A_i+\eta\delta B_X)\subset\Big(\bigcap_{i\in I}A_i\Big)+\delta B_X\ \ \forall \delta>0.
\end{equation}

From the definition, the uniform normal property implies the normal property and $\{A_i: i\in I\}$ has the closed intersection property if each $A_i$ is closed.\\

\noindent{\bf Remark 3.1.} When $I$ is a finite index set, saying $I:=\{1, \cdots, m\}$, the weak normal property in Definition 3.1(iii) is equivalent to that for any $x^*\in X^*$ there is one constant $\eta^{\prime}_{x^*}>0$ such that
\begin{equation}\label{3.5a}
\bigcap_{i=1}^m(A_i+\eta^{\prime}_{x^*} B_X)\subset\Big(\bigcap_{i=1}^mA_i\Big)+\{x^*\}^{\circ}.
\end{equation}
Indeed, we only need to show that \eqref{3.3} implies \eqref{3.5a}. Let $x^*\in X^*$ and suppose that \eqref{3.3} holds with $\eta_{x^*}>0$. Since
$$
A_i+\frac{\eta_{x^*}}{2}B_X\subset A_i+{\rm int}(\eta_{x^*} B_X)\subset {\rm int}(A_i+\eta_{x^*} B_X),
$$
one has
\begin{eqnarray*}
\bigcap_{i=1}^m \Big(A_i+\frac{\eta_{x^*}}{2}B_X\Big)&\subset&\bigcap_{i=1}^m {\rm int}(A_i+\eta_{x^*} B_X)\subset \Big({\rm int}\bigcap_{i=1}^m(A_i+\eta_{x^*} B_X)\Big)\\
&\subset&{\rm int}\left(\Big(\overline{\bigcap_{i=1}^mA_i\Big)+\{x^*\}^{\circ}}\right).
\end{eqnarray*}
Noting that $(x^{*})^{-1}(-\infty, 1)\subset{\rm int}\big(\{x^*\}^{\circ}\big)$ and thus ${\rm int}\big(\{x^*\}^{\circ}\big)\neq\emptyset$, by Lemma 2.3, it follows that
$$
\bigcap_{i=1}^m \Big(A_i+\frac{\eta_{x^*}}{2}B_X\Big)\subset{\rm int}\left(\Big(\overline{\bigcap_{i=1}^mA_i\Big)+\{x^*\}^{\circ}}\right)={\rm int}\left(\Big(\bigcap_{i=1}^mA_i\Big)+\{x^*\}^{\circ}\right).
$$
Hence \eqref{3.5a} holds with $\eta^{\prime}_{x^*}:=\frac{\eta_{x^*}}{2}$.

The following proposition, similar to [3, Theorem 3.1], shows the
interrelationship of various normal properties for an infinite
collection of convex cones. We give its proof for sake of
completeness.

\begin{pro} Let $\{K_i: i\in I\}$ be a collection of convex cones
in $X$.

$\rm (i)$ $\{K_i: i\in I\}$ has the normal property if and only if
it has the uniform normal property.

$\rm (ii)$ If $\{K_i: i\in I\}$ has the normal property, then it has
the closed intersection property and the weak normal property.

$\rm (iii)$ $\{K_i: i\in I\}$ has the normal property (resp. the weak
normal property) if and only if it has the closed intersection
property and $\{\overline K_i: i\in I\}$ has the normal property
(resp. the weak normal property).
\end{pro}
\begin{proof} (i) is immediate from Definition 3.1.

(ii) Since $\{K_i: i\in I\}$ has the normal property, by (i), there
exists $\eta>0$ such that
\begin{equation}\label{3.5}
\bigcap_{i\in I}(K_i+\eta\delta B_X)\subset\Big(\bigcap_{i\in I}K_i\Big)+\delta B_X\ \ \forall \delta>0.
\end{equation}
This implies that for any $\delta>0$, one has
$$
\bigcap_{i\in I}\overline K_i\subset \bigcap_{i\in I}(K_i+\eta\delta
B_X)\subset \Big(\bigcap_{i\in I}K_i\Big)+\delta B_X.
$$
Thus the closed intersection property for $\{K_i: i\in I\}$ holds.

To verify the weak normal property, let $x^*\in X^*$. If
$\|x^*\|\leq 1$, then $B_X\subset \{x^*\}^{\circ}$ and consequently
inclusion \eqref{3.3} for $\{K_i: i\in I\}$ holds with
$\eta_{x^*}:=\eta$ by \eqref{3.5}. If $\|x^*\|>1$, by using
\eqref{3.5}, one has
$$
\bigcap_{i\in I}\Big(K_i+\frac{\eta}{\|x^*\|}
B_X\Big)\subset\Big(\bigcap_{i\in I}K_i\Big)+\frac{1}{\|x^*\|}B_X.
$$
Noting that $\frac{1}{\|x^*\|}B_X\subset \{x^*\}^{\circ}$, it
follows that inclusion \eqref{3.3} for $\{K_i: i\in I\}$ holds with
$\eta_{x^*}:=\eta/\|x^*\|$. Hence $\{K_i: i\in I\}$ has the weak
normal property.

(iii) Suppose that $\{K_i: i\in I\}$ has the normal property. Then
the closed intersection property holds by (ii). Using \eqref{3.2},
there exists $\eta>0$ such that
$$
\bigcap_{i\in I}(\overline K_i+\frac{\eta}{2}
B_X)\subset\bigcap_{i\in I}(K_i+\eta B_X)\subset\Big(\bigcap_{i\in
I}K_i\Big)+B_X\subset\Big(\bigcap_{i\in I}\overline K_i\Big)+B_X.
$$
This means that $\{\overline K_i: i\in I\}$ has the normal property.

Conversely, there exists $\eta>0$ such that \eqref{3.2} holds for
$\{\overline K_i: i\in I\}$. From this, we can derive that
$$
\bigcap_{i\in I}(K_i+\eta B_X)\subset\bigcap_{i\in I}(\overline
K_i+\eta B_X)\subset\Big(\bigcap_{i\in I}\overline K_i\Big)+B_X.
$$
By virtue of the closed intersection property and Lemma 2.2, one has
$$
\bigcap_{i\in I}(K_i+\eta B_X)\subset\overline{\bigcap_{i\in
I}K_i}+B_X\subset\overline{\Big(\bigcap_{i\in I}
K_i\Big)+B_X}\subset\Big(\bigcap_{i\in I} K_i\Big)+2B_X.
$$
By multiplying both sides by $1/2$, we have that $\{K_i: i\in I\}$
has the normal property with $\eta/2>0$ since each $K_i$ is a cone.

It remains to prove (iii) concerning the weak normal property.
Suppose that $\{K_i: i\in I\}$ has the weak normal property. Then
for each $x^*\in X^*$ there exists $\eta_{x^*}>0$ such that
$$
\bigcap_{i\in I}(\overline K_i+\frac{\eta_{x^*}}{2} B_X)\subset
\bigcap_{i\in I}(K_i+\eta_{x^*}B_X)\subset \overline{\Big(\bigcap_{i\in I}
K_i\Big)+\{x^*\}^{\circ}}\subset\overline{\Big(\bigcap_{i\in I} \overline
K_i\Big)+\{x^*\}^{\circ}}.
$$
Thus $\{\overline K_i: i\in I\}$ has the weak normal property with
$\eta_{x^*}/2>0$ for each $x^*\in X^*$.

Now, we prove the closed intersection property. We claim that
\begin{equation}\label{3.6}
\bigcap_{x^*\in X^*}\Big(\overline{\big(\bigcap_{i\in
I}K_i\big)+\{x^*\}^{\circ}}\Big)\subset \overline{\bigcap_{i\in
I}K_i}.
\end{equation}
Granting this, the closed intersection property follows from the
following inclusions
$$
\bigcap_{i\in I}\overline K_i\subset\bigcap_{x^*\in
X^*}\Big(\bigcap_{i\in I}(K_i+\eta_{x^*}B_X
)\Big)\subset\bigcap_{x^*\in X^*}\Big(\overline{\big(\bigcap_{i\in
I}K_i\big)+\{x^*\}^{\circ}}\Big).
$$

Suppose to the contrary that there exists $x_0\in X$ such that
\begin{equation}\label{3.7}
x_0\in\bigcap_{x^*\in X^*}\Big(\overline{\big(\bigcap_{i\in
I}K_i\big)+\{x^*\}^{\circ}}\Big)\Big\backslash
\Big(\overline{\bigcap_{i\in I}K_i}\Big).
\end{equation}
By the seperation theorem, there exist $\tilde{x}^*_0\in X^*$ with
$\|\tilde{x}^*_0\|=1$ and $\beta\in \mathbb{R}$ such that
\begin{equation}\label{3.8}
\langle \tilde{x}^*_0, x_0\rangle>\beta>\sup\Big\{\langle
\tilde{x}^*_0, z\rangle: z\in \overline{\bigcap_{i\in I}K_i}\Big\}.
\end{equation}
From $0\in \overline{\bigcap\limits_{i\in I}K_i}$, one has $\beta>0$ and
\eqref{3.8} can be rewritten as
\begin{equation}\label{3.9}
\langle x^*_0, x_0\rangle>1>\sup\Big\{\langle x^*_0, z\rangle: z\in
\overline{\bigcap_{i\in I}K_i}\Big\}
\end{equation}
where $x^*_0:=\tilde{x}^*_0/\beta$. Noting that $\bigcap_{i\in
I}K_i$ is a cone, it follows from \eqref{3.9} that
$$
\langle x^*_0, z\rangle\leq 0\ \ \forall z\in\bigcap_{i\in I}K_i.
$$
This and \eqref{3.9} imply that
$$
\langle x^*_0, x_0-z\rangle> 1\ \ \forall z\in\bigcap_{i\in I}K_i.
$$
Thus $x_0\not\in \overline{\big(\bigcap\limits_{i\in
I}K_i\big)+\{x_0^*\}^{\circ}}$, which contradicts \eqref{3.7}. Hence
\eqref{3.6} holds.

Next, we suppose that $\{\overline K_i: i\in I\}$ has the weak
normal property and the closed intersection property. Then for each
$x^*\in X^*$, there exists $\eta_{x^*}>0$ such that
$$
\bigcap_{i\in I}(\overline K_i+\eta_{x^*}B_X)\subset
\overline{\Big(\bigcap_{i\in I}\overline K_i\Big)+\{x^*\}^{\circ}}.
$$
This implies that
$$
\bigcap_{i\in I}(K_i+\eta_{x^*}B_X)\subset \overline{\Big(\bigcap_{i\in
I}\overline K_i\Big)+\{x^*\}^{\circ}}=\overline{\Big(\overline{\bigcap_{i\in
I}K_i}\Big)+\{x^*\}^{\circ}} = \overline{\Big(\bigcap_{i\in
I}K_i\Big)+\{x^*\}^{\circ}}
$$
where the first equality is by the closed intersection and the second equality
follows from Lemma 2.2. Hence $\{K_i: i\in I\}$ has the
weak normal property. The proof is completed.
\end{proof}

\setcounter{equation}{0}
\section{Dual normality and Jamenson property (property ($G$))}
In this section, we study a dual form of the normal property, and
provide a quantitative relationship between the normal property and
the dual normal property. Then, by the notion of extended Jamenson
property for a collection of arbitrarily many weak$^*$-closed
convex cones in $X^*$ introduced in [30], we establish the
equivalence between the normal property and its dual normality form
for the infinite system of arbitrarily many convex
cones.

The main tool used to study the dual normal property and the
extended Jamenson property is the {\it inverse sum} of two convex
sets. We first recall the definition of
inverse sum of two convex sets.\\

\noindent{\bf Definition 4.1.} Let $A_1$ and $A_2$ be two convex sets of $X$ with
$0\in A_1\cap A_2$. The {\it inverse sum} of $A_1$ and $A_2$ is
defined by
\begin{equation}\label{4.1}
A_1\#A_2:=\Big(\bigcup_{0<t<1}(tA_1\cap
(1-t)A_2)\Big)\cup(A_1\cap 0^+A_2)\cup(A_2\cap 0^+A_1).
\end{equation}

The following proposition gives some properties about inverse sum of
convex sets which will be used later in our analysis. Readers are
invited to see [3, Lemma 4.1] for more details and its proof.

\begin{pro} Let $A_1$ and $A_2$ be convex sets of $X$ such that
$0\in A_1\cap A_2$. Then

(a) $(A_1+A_2)^{\circ}=A_1^{\circ}\#A_2^{\circ}$.

(b) If $A_1$ is cone, then $A_1\#A_2=A_1\cap A_2$.

(c) $A_1\#A_2$ is a convex set. Furthermore, if $A_1$ and $A_2$ are closed,
then $A_1\#A_2$ is closed.
\end{pro}

Now, we use the inverse sum to state the following theorem on
the dual form of the weak normal property.

\begin{them}
Let $\{A_i: i\in I\}$ be a collection of convex sets in $X$ such
that $0\in \bigcap_{i\in I}A_i$. Suppose that $\{A_i: i\in I\}$ has
the closed intersection property. Then $\{A_i: i\in I\}$ has the
weak normal property if and only if for each $x^*\in X^*$ there
exists $\tilde{\eta}_{x^*}>0$ such that
\begin{equation}\label{4.2}
[0, x^*]\#\overline{\rm co}^{w^*}\Big(\bigcup\limits_{i\in I}A_i^{\circ}\Big)\subset \overline{\rm
co}^{w^*}\Big(\bigcup_{i\in
I}\big(A_i^{\circ}\#(\tilde{\eta}_{x^*}B_{X^*})\big)\Big)
\end{equation}
where $[0, x^*]:=\{tx^*: t\in [0, 1]\}$.
\end{them}
\begin{proof}
{\bf The necessity part}. Since $\{A_i: i\in I\}$ has the weak normal
property, then for each $x^*\in X^*$ there exists
$\eta_{x^*}>0$ such that
\begin{equation}\label{4.3}
\bigcap_{i\in I}(A_i+\eta_{x^*} B_X)\subset\overline{\bigcap_{i\in
I}A_i+\{x^*\}^{\circ}}.
\end{equation}
Let $\eta:=\eta_{x^*}/2$. By \eqref{4.3}, one has
\begin{equation}\label{4.4}
\bigcap_{i\in I}\overline{(A_i+\eta B_X)}\subset\overline{\bigcap_{i\in
I}A_i+\{x^*\}^{\circ}}.
\end{equation}
Applying the polar operation to both side of \eqref{4.4}, by Lemma
2.1(a) and (e), we obtain
\begin{equation}\label{4.5}
\Big(\overline{\bigcap_{i\in
I}A_i+\{x^*\}^{\circ}}\Big)^{\circ}\subset\Big(\bigcap_{i\in I}\overline{(A_i+\eta B_X)}\Big)^{\circ}=\overline{\rm co}^{w^*}\Big(\bigcup_{i\in I}(A_i+\eta B_X)^{\circ}\Big)
\end{equation}
Noting that $\{A_i: i\in I\}$ has the closed intersection property,
it follows from Lemma 2.1(a) and (e) that
\begin{equation}\label{4.6}
\Big(\bigcap_{i\in I}A_i\Big)^{\circ}=\Big(\overline{\bigcap_{i\in I}A_i}\Big)^{\circ}=\Big(\bigcap_{i\in I}\overline A_i\Big)^{\circ}=\overline{\rm co}^{w^*}\Big(\bigcup_{i\in I}A_i^{\circ}\Big)
\end{equation}
From this, using Lemma 2.1(d) and Proposition 4.1(a), we get
\begin{eqnarray}\label{4.7}
\Big(\bigcap\limits_{i\in
I}A_i+\{x^*\}^{\circ}\Big)^{\circ}
=\{x^*\}^{\circ\circ}\#\Big(\bigcap\limits_{i\in I}A_i\Big)^{\circ}=[0, x^*]\#\overline{\rm co}^{w^*}\Big(\bigcup\limits_{i\in I}A_i^{\circ}\Big).
\end{eqnarray}
Noting that $(A_i+\eta B_X)^{\circ}=A_i^{\circ}\#\frac{1}{\eta}B_{X^*}$ (by Proposition 4.1(a) and Lemma 2.1(c)), it follows from Lemma 2.1(a), (c), \eqref{4.5} and \eqref{4.7} that
$$
[0, x^*]\#\overline{\rm co}^{w^*}\Big(\bigcup\limits_{i\in I}A_i^{\circ}\Big)\subset \overline{\rm co}^{w^*}\Big(\bigcup_{i\in I}(A_i^{\circ}\#\frac{1}{\eta}B_{X^*})\Big).
$$
Hence \eqref{4.2} holds for $\tilde{\eta}_{x^*}:=1/\eta=2/\eta_{x^*}>0$.

{\bf The sufficiency part}. Let $x^*\in X^*$. Then there exists $\tilde{\eta}_{x^*}>0$ such that \eqref{4.2} holds. Let $\eta:=1/\tilde{\eta}_{x^*}$. Note that
\begin{eqnarray*}
\bigcap_{i\in I}(A_i+\eta B_X)\subset \Big(\bigcap_{i\in I}(\overline{A_i+\eta B_X})\Big)^{\circ\circ}=\left(\overline{\rm co}^{w^*}\Big(\bigcup_{i\in I}(A_i^{\circ}\#\frac{1}{\eta}B_{X^*})\Big)\right)^{\circ}
\end{eqnarray*}
where the inclusion is by Lemma 2.1(d) and the equality follows from Lemma 2.1(e) and Proposition 4.1(a). Taking the polar to both sides of \eqref{4.2}, we have
\begin{eqnarray*}
\bigcap_{i\in I}(A_i+\eta B_X)&\subset& \left(\overline{\rm co}^{w^*}\Big(\bigcup_{i\in I}(A_i^{\circ}\#\frac{1}{\eta}B_{X^*})\Big)\right)^{\circ}\subset \Big([0, x^*]\#\overline{\rm co}^{w^*}\Big(\bigcup_{i\in I}A_i^{\circ}\Big)\Big)^{\circ}\\
&=&\Big(\bigcap_{i\in I}A_i+\{x^*\}^{\circ}\Big)^{\circ\circ}=\overline{\bigcap_{i\in I}A_i+\{x^*\}^{\circ}}
\end{eqnarray*}
where the first equality follows from \eqref{4.7} and the second equality holds by Lemma 2.1(d). This means that $\{A_i: i\in I\}$ has the weak normal property with $\eta_{x^*}:=1/\tilde{\eta}_{x^*}$ for each $x^*\in X^*$. The proof is completed.
\end{proof}

By using the inverse sum of convex sets, we obtain the quantitative relationship between the normal property and its dual form through the following theorem.

\begin{them}
Let $\{A_i: i\in I\}$ be a collection of convex sets in $X$ with
$0\in \bigcap_{i\in I}A_i$ and let $\eta>0$. Suppose that $\{A_i:
i\in I\}$ has the closed intersection property. If
\begin{equation}\label{4.8}
\bigcap_{i\in I}(A_i+\eta B_X)\subset \Big(\bigcap_{i\in I}A_i\Big)+B_X,
\end{equation}
then for any $\hat{\eta}\in (0, \eta)$, one has
\begin{equation}\label{4.9}
B_{X^*}\#\overline{\rm co}^{w^*}\Big(\bigcup_{i\in I}A_i^{\circ}\Big)\subset \overline{\rm co}^{w^*}\Big(\bigcup_{i\in I}(A_i^{\circ}\#\frac{1}{\hat{\eta}}B_{X^*})\Big).
\end{equation}
Conversely, if \eqref{4.9} holds, then
\begin{equation}\label{4.10}
\bigcap_{i\in I}(A_i+\hat{\eta} B_X)\subset \overline{\Big(\bigcap_{i\in
I}A_i\Big)+B_X}.
\end{equation}
\end{them}
\begin{proof} Let $\hat{\eta}\in (0, \eta)$. Then \eqref{4.8} implies that
$$
\bigcap_{i\in I}(\overline{A_i+\hat{\eta} B_X})\subset\bigcap_{i\in I}(A_i+\eta B_X)\subset \Big(\bigcap_{i\in I}A_i\Big)+B_X.
$$
Taking the polar, we obtain
\begin{equation}\label{4.11}
\Big(\big(\bigcap_{i\in I}A_i\big)+B_X\Big)^{\circ}\subset\Big(\bigcap_{i\in I}(\overline{A_i+\hat{\eta} B_X})\Big)^{\circ}
\end{equation}
By using Lemma 2.1(e), Proposition 4.1(a) and \eqref{4.6}, we have
\eqref{4.9} holds.

Conversely, assume that \eqref{4.9} holds. Noting that $\{A_i: i\in I\}$ has the closed intersection property, it follows from \eqref{4.6}, Proposition 4.1(a) and Lemma 2.1(d) that
$$
\Big(B_{X^*}\#\overline{\rm co}^{w^*}\big(\bigcup_{i\in I}A^{\circ}_i\big)\Big)^{\circ}=\Big(B_X+\bigcap_{i\in I}A_i\Big)^{\circ\circ}=\overline{B_X+\bigcap_{i\in I}A_i}.
$$
This implies that \eqref{4.10} holds. The proof is completed.
\end{proof}

Now, we give the quantitative measurements of the normal property and
its dual form for arbitrarily many convex sets.\\

\noindent{\bf Definition 4.2.} Let $\{A_i: i\in I\}$ be a collection of
convex sets in $X$. The {\it normality constant} for the collection
$\{A_i: i\in I\}$ is defined as follows:
\begin{equation}\label{4.12}
\lambda_N(A_i:i\in I):=\sup\{\eta\geq 0: \eqref{4.8}\ {\rm holds}\
{\rm with}\ \eta\geq 0\},
\end{equation}
and the {\it dual normality constant} for the collection
$\{A_i^{\circ}: i\in I\}$ is defined by
\begin{equation*}\label{4.13}
\lambda_D(A^{\circ}_i:i\in I):=\sup\left\{\eta\geq 0:
B_{X^*}\#\overline{\rm co}^{w^*}\Big(\bigcup_{i\in
I}A_i^{\circ}\Big)\subset \overline{\rm co}^{w^*}\Big(\bigcup_{i\in
I}(A_i^{\circ}\#\frac{1}{\eta}B_{X^*})\Big)\right\}
\end{equation*}
where $\frac{1}{\eta}B_{X^*}$ is defined to be the whole space $X^*$
when $\eta=0$.

Obviously $\{A_i: i\in I\}$ has the normal property if and only if
$\lambda_N(A_i: i\in I)>0$. The following corollary shows the precise
quantitative equation for the collection of arbitrarily many convex
cones with the closed intersection property.

\begin{coro}
Let $\{K_i: i\in I\}$ be a collection of convex cones in $X$ and have the closed intersection property. Then
\begin{equation}\label{4.14}
\lambda_N(K_i:i\in I)=\lambda_D(K^{\circ}_i:i\in I).
\end{equation}
\end{coro}
\begin{proof}
From the implication \eqref{4.8}$\Rightarrow$\eqref{4.9} in Theorem 4.2, one has
\begin{equation}\label{4.15}
0\leq \lambda_N(K_i:i\in I)\leq\lambda_D(K^{\circ}_i:i\in I).
\end{equation}
If $\lambda_D(K^{\circ}_i:i\in I)=0$, then \eqref{4.14} holds. Next,
we suppose that $\lambda_D(K^{\circ}_i:i\in I)>0$. Take arbitrarily $\eta\in (0, \lambda_D(K^{\circ}_i:i\in I))$ and let
$\varepsilon>0$. It follows from
\eqref{4.9}$\Rightarrow$\eqref{4.10} that
$$
\bigcap_{i\in I}(K_i+\eta B_X)\subset \overline{\Big(\bigcap_{i\in I}K_i\Big)+B_X}\subset\Big(\bigcap_{i\in I}K_i\Big)+(1+\varepsilon)B_X.
$$
This implies that
$$
\bigcap_{i\in I}\big(K_i+\frac{\eta}{1+\varepsilon}
B_X\big)\subset\Big(\bigcap_{i\in I}K_i\Big)+B_X
$$
as each $K_i$ is a cone and consequently $\lambda_N(K_i:i\in I)\geq\frac{\eta}{1+\varepsilon}$ by \eqref{4.12}. Taking limits as $\varepsilon\rightarrow 0^+$ and $\eta\rightarrow \lambda_D(K^{\circ}_i:i\in I)$, we have $\lambda_N(K_i:i\in I)\geq \lambda_D(K^{\circ}_i:i\in I)$. Hence \eqref{4.14} holds by \eqref{4.15}. The proof is completed.
\end{proof}

Let $\eta>0$. Recall from [7, 9, 20] that subsets $P_1, \cdots, P_m$
in $X$ is said to have Jamenson property $(G_{\eta})$, if
\begin{equation}
\Big(\sum_{i=1}^mP_i\Big)\cap B_X\subset \sum_{i=1}^m(P_i\cap
\frac{1}{\eta}B_X).
\end{equation}
In the case when each $P_i$ is a weak$^*$-closed convex cone in the
dual space $X^*$, the authors [30] obtained one important
characterization for Jamenson property ($G_{\eta}$); that is,
$\{P_1, \cdots, P_m\}$ in $X$ is said to have property
$(G_{\eta})$ if and only if for each $x^*\in
\overline{\sum_{i=1}^mP_i}^{w^*}$ there exist generalized sequences
$\{x_{i, k}^*\}$ in $P_i$ $(1\leq i\leq m)$  such that
\begin{equation}\label{4.17}
\sum_{i=1}^m x_{i, k}^*\stackrel{w^*}\longrightarrow x^*\ \ {\rm
and}\ \ \sum_{i=1}^m\|x_{i, k}^*\|\leq \frac{1}{\eta}\|x^*\|.
\end{equation}
Readers could refer to [30, Proposition 5.1] for details and the
proof.

In the spirit of the equivalence in \eqref{4.17}, the authors [30]
introduced and studied the extended Jamenson property $(G_{\eta})$ for a
collection of arbitrarily many weak$^*$-closed convex cone in the
dual space $X^*$.\\

\noindent{\bf Definition 4.3.} Let $\{P_i: i\in I\}$ be a collection of
weak$^*$-closed convex cone in $X^*$ and $\eta>0$. We say that
$\{P_i: i\in I\}$ has {\it property $(G_{\eta})$}, if for any
$x^*\in\overline{\sum\limits_{i\in I}P_i}^{w^*}$ there exist a
generalized sequences $\{I_k\}$ of finite subsets of $I$ and
$x^*_{i, k}\in P_i (i\in I_k)$ such that
\begin{equation}\label{4.18}
\sum_{i\in I_k}x_{i, k}^*\stackrel{w^*}\longrightarrow x^*\ \ {\rm
and}\ \ \sum_{i\in I_k}\|x_{i, k}^*\|\leq \frac{1}{\eta}\|x^*\|.
\end{equation}

By the extended Jamenson property $(G_{\eta})$, we obtain the
following theorem about the characterization of normal property
for an infinite system of convex cones in a Banach space.

\begin{them}
Let $\{K_i: i\in I\}$ be a collection of convex cones in $X$. Then
$\{K_i: i\in I\}$ has the normal property if and only if $\{K_i:
i\in I\}$ has the closed intersection property and there exists
$\eta>0$ such that $\{K_i^{\circ}: i\in I\}$ has property
$(G_{\eta})$.
\end{them}
\begin{proof}
{\bf The necessity part.} Suppose that $\{K_i: i\in I\}$ has the
normal property. Then the closed intersection property of $\{K_i:
i\in I\}$ is immediate from Proposition 3.1(ii), and there exists
$\hat{\eta}>0$ such that
\begin{equation}\label{4.19}
\bigcap_{i\in I}(K_i+\hat{\eta} B_X)\subset\Big(\bigcap_{i\in
I}K_i\Big)+ B_X.
\end{equation}
Let $\eta\in (0, \hat{\eta})$ and $x^*\in
\overline{\sum\limits_{i\in
I}K^{\circ}_i}^{w^*}\big\backslash\{0\}$. By Theorem 4.2, we obtain
\begin{equation}\label{4.20}
B_{X^*}\#\overline{\rm co}^{w^*}\Big(\bigcup_{i\in
I}K_i^{\circ}\Big)\subset \overline{\rm co}^{w^*}\Big(\bigcup_{i\in
I}(K_i^{\circ}\#\frac{1}{\eta}B_{X^*})\Big),
\end{equation}
i.e.,
\begin{equation}\label{4.21}
B_{X^*}\cap\Big(\overline{\sum\limits_{i\in
I}K^{\circ}_{i}}^{w^*}\Big)\subset \overline{\rm
co}^{w^*}\Big(\bigcup_{i\in
I}(K_i^{\circ}\cap\frac{1}{\eta}B_{X^*})\Big)
\end{equation}
Noting that $x^*\in \overline{\sum\limits_{i\in
I}K^{\circ}_{i}}^{w^*}\big\backslash\{0\}$, it follows from
\eqref{4.21} that there exist a generalized sequences
$\{\tilde{x}_k^*\}$ in ${\rm co}\big(\bigcup_{i\in
I}(K_i^{\circ}\cap\frac{1}{\eta}B_{X^*})\big)$ such that
$\tilde{x}_k^*\stackrel{w^*}\longrightarrow \frac{x^*}{\|x^*\|}$.
Thus, for each $k$, there are a finite subset $I_k$ of $I$,
nonnegative scalars $\mu_{i, k}\geq 0$ and $\tilde x_{i, k}^*\in
K_i^{\circ}\cap \frac{1}{\eta}B_{X^*}$ $(\forall i\in I_k)$ such
that
\begin{equation}\label{4.22}
\sum_{i\in I_k}\mu_{i, k}=1\ \ {\rm and}\ \ \tilde{x}_k^*=\sum_{i\in
I_k}\mu_{i, k}\tilde{x}^*_{i, k}.
\end{equation}
Let $x_{i, k}^*:=\mu_{i, k}\|x^*\|\cdot\tilde{x}^*_{i, k}$ for each
$i\in I_k$. Then \eqref{4.22} implies that $x_{i, k}^*\in
K^{\circ}_i (\forall i\in I_k)$,
$$
\sum_{i\in I_k}x^*_{i, k}\stackrel{w^*}\longrightarrow x^*\ \ {\rm
and}\ \ \sum_{i\in I_k}\|x^*_{i, k}\|\leq \frac{1}{\eta}\|x^*\|.
$$
Hence $\{K^{\circ}_i: i\in I\}$ has property $(G_{\eta})$.

{\bf The sufficiency part.} Take $\eta>0$ such that $\{K_i^{\circ}:
i\in I\}$ has property $(G_{\eta})$. We claim that \eqref{4.20} and
\eqref{4.21} hold.

Indeed, for any $x^*\in \overline{\sum\limits_{i\in
I}K^{\circ}_i}^{w^*}\big\backslash\{0\}$ with $\|x^*\|\leq 1$, by
property $(G_{\eta})$, there exists a generalized sequences
$\{I_k\}$ of finite subsets of $I$ and $x^*_{i, k}\in
K_i^{\circ}\backslash\{0\} (i\in I_k)$ such that
\begin{equation}\label{4.23}
\sum_{i\in I_k}x_{i, k}^*\stackrel{w^*}\longrightarrow x^*\ \ {\rm
and}\ \ \sum_{i\in I_k}\|x_{i, k}^*\|\leq \frac{1}{\eta}\|x^*\|\leq
\frac{1}{\eta}.
\end{equation}
Denote
$$
y_{i, k}^*:=\frac{\sum\limits_{j\in I_k}\|x^*_{j, k}\|}{\|x^*_{i,
k}\|}x^*_{i, k}\ \ {\rm and}\ \ \mu_{i, k}:=\frac{\|x^*_{i,
k}\|}{\sum\limits_{j\in I_k}\|x^*_{j, k}\|}\ \ \forall i\in I_k.
$$
Then \eqref{4.23} implies that $y_{i, k}^*\in
K_i^{\circ}\cap\frac{1}{\eta}B_{X^*}$ and $\sum\limits_{i\in
I_k}x^*_{i, k}=\sum\limits_{i\in I_k}\mu_{i, k}y^*_{i, k}$. Thus
\eqref{4.21} holds and so does \eqref{4.20}.

Let $\hat{\eta}\in (0, \eta)$. Since $\{K_i: i\in I\}$ has the
closed intersection property, it follows from \eqref{4.20} and the
implication \eqref{4.9}$\Rightarrow$\eqref{4.10} in Theorem 4.2 that
\begin{equation}\label{4.24}
\bigcap_{i\in I}(K_i+\eta B_X)\subset \overline{\Big(\bigcap_{i\in
I}K_i\Big)+B_X}\subset\Big(\bigcap_{i\in
I}K_i\Big)+\frac{\eta}{\hat{\eta}}B_X.
\end{equation}
By multiplying both side of \eqref{4.24} by $\hat{\eta}/\eta$, we
obtain
$$
\bigcap_{i\in I}(K_i+\hat{\eta} B_X)\subset \Big(\bigcap_{i\in
I}K_i\Big)+B_X
$$
since each $K_i$ is a cone. Therefore, $\{K_i: i\in I\}$ has the
normal property with the constant $\hat{\eta}$. The proof is
completed.
\end{proof}

For the case that $\{P_i: i\in I\}$ is a collection of
weak$^*$-closed convex cones in $X^*$, we define the quantitative
measurement of the extended Jamenson property as follows:
\begin{equation*}\label{4.25}
\lambda_G(P_i:i\in I):=\sup\Big\{\eta\geq 0: \{P_i: i\in I\} \ {\rm
has} \ {\rm property}\  (G_{\eta})\Big\}.
\end{equation*}

The following corollary is immediate from Theorem 4.3 and Corollary
4.1 which shows the quantitative equations among measurements of the
normal property, the dual normality and the extended Jamenson
property.

\begin{coro}
Let $\{K_i: i\in I\}$ be a collection of convex cones in $X$ with
the closed intersection property. Then
$$
\lambda_D(K^{\circ}_i:i\in I)=\lambda_N(K_i:i\in
I)=\lambda_G(K^{\circ}_i:i\in I).
$$
\end{coro}

\setcounter{equation}{0}
\section{CHIP, normal CHIP and linear regularity} In this section,
we mainly focus on CHIP, normal CHIP and linear regularity, and
establish the equivalent interrelationship among CHIP, linear regularity,
normal property and Jamenson property. The notions of CHIP and strong CHIP were introduced by
Chui, Deustsch and Ward [10] and Deutsch, Li and Ward [15], respectively and have been studied extensively by many authors(cf. [3, 7, 8, 12, 14, 19, 21, 30] and references therein). For the further study on CHIP and strong CHIP,
Bakan, Deutsch and Li [3] introduced the concept of normal conical
hull intersection property (normal CHIP) for finite convex sets in a
Hilbert space. We begin by recalling definitions on several kinds of CHIP for a
collection of arbitrarily many convex sets in a Banach space.\\

\noindent{\bf Definition 5.1.} Let $\{A_i: i\in I\}$ be a collection of
convex sets in $X$ with the nonempty intersection.

(i) $\{A_i: i\in I\}$ is said to have the {\it conical hull
intersection property (CHIP)} at $x\in\bigcap_{i\in I}A_i$, if
\begin{equation}\label{5.1}
\overline{\rm cone}\Big(\bigcap_{i\in I}A_i-x\Big)=\bigcap_{i\in
I}\overline{\rm cone}(A_i-x).
\end{equation}
We say that $\{A_i: i\in I\}$ has the {\it CHIP}, if it has the CHIP
for each $x\in\bigcap_{i\in I}A_i$.

(ii)$\{A_i: i\in I\}$ is said to have the {\it strong conical hull
intersection property (strong CHIP)} at $x\in\bigcap_{i\in I}A_i$,
if
\begin{equation}\label{5.2}
\Big(\bigcap_{i\in I}A_i-x\Big)^{\ominus}=\sum_{i\in
I}(A_i-x)^{\ominus}.
\end{equation}
We say that  $\{A_i: i\in I\}$ has the {\it strong CHIP}, if it has
the strong CHIP for each $x\in\bigcap_{i\in I}A_i$.

(iii) $\{A_i: i\in I\}$ is said to have the {\it normal conical hull
intersection property (normal CHIP)} at $x\in\bigcap_{i\in I}A_i$,
if the collection of convex cones
\begin{center}
 $\big\{{\rm cone}(A_i-x): i\in
I\big\}$
\end{center}
has the normal property. We say that  $\{A_i: i\in I\}$ has the {\it
normal CHIP}, if it has the normal CHIP for each $x\in\bigcap_{i\in
I}A_i$.

(iv) $\{A_i: i\in I\}$ is said to have the {\it weak normal conical
hull intersection property (weak normal CHIP)} at $x\in\bigcap_{i\in
I}A_i$, if the collection of convex cones
\begin{center}
 $\big\{{\rm cone}(A_i-x): i\in
I\big\}$
\end{center}
has the weak normal property. We say that  $\{A_i: i\in I\}$ has the
{\it weak normal CHIP}, if it has the weak normal CHIP for each
$x\in\bigcap_{i\in I}A_i$.\\

\noindent{\bf Remark 5.1.} (a) It is easy to verify from the definition that
\begin{center}
$\{A_i: i\in I\}$ has the CHIP at $x\in\bigcap\limits_{i\in I}A_i$\\
$\Longleftrightarrow$ $T\Big(\bigcap\limits_{i\in I}A_i,
x\Big)=\bigcap\limits_{i\in I} T(A_i, x)$\ \ \ \ \ \ \ \ \ \ \ \ \ \
\ \ \ \\
$\Longleftrightarrow$ $N\Big(\bigcap\limits_{i\in I}A_i,
x\Big)=\overline{\sum\limits_{i\in I}N(A_i, x)}^{w^*}$,\ \ \ \ \ \ \
\ \ \ \ \ \ \
\end{center}
and $\{A_i: i\in I\}$ has the strong CHIP at $x$ if and only if
$\{A_i: i\in I\}$ has the CHIP at $x$ and $\sum_{i\in I}N(A_i, x)$
is weak$^*$-closed.

(b) In the case when $I$ is finite and $X$ is a Hilbert space,
Bakan, Deutsch and Li [3] gave the definition of CHIP through the
closed intersection property rather than by \eqref{5.1}; that is
$\{A_i: i\in I\}$ is said to have the CHIP at $x\in\bigcap_{i\in
I}A_i$, if the collection of convex cones
\begin{center}
 $\big\{{\rm cone}(A_i-x): i\in
I\big\}$
\end{center}
has the closed intersection property, i.e., if
\begin{equation}\label{5.3}
\overline{\bigcap_{i\in I}{\rm cone}(A_i-x)}=\bigcap_{i\in
I}\overline{\rm cone}(A_i-x).
\end{equation}
Obviously the implication \eqref{5.1}$\Rightarrow$\eqref{5.3} holds
trivially and for this case where $I$ is finite and $X$ is a Hilbert
space \eqref{5.1} is equivalent to \eqref{5.3} since the union can
commute with the intersection in this case (see [3, Lemma 2.2] for
details). However, if the index set $I$ is infinite, \eqref{5.3}
would not imply \eqref{5.1} necessarily even in the
finite-dimensional space. For example, let $X:=\mathbb{R}$,
$I:=\mathbb{N}$ and $A_i:=(-\frac{1}{i},\frac{1}{i})$ for each $i\in
I$. Then $\{A_i: i\in I\}$ is a collection of convex sets in $X$ and
$\bigcap_{i\in I}A_i=\{0\}$. By the computation, one has
$$
{\rm cone}\Big(\bigcap_{i\in I}A_i-0\Big)=\{0\}\ \ {\rm and}\ \ {\rm
cone}(A_i-0)=X\ \ \forall i\in I.
$$
This implies that \eqref{5.3} holds but \eqref{5.1} does not. This
example also shows that \eqref{5.3} fails to characterize the
original notion of CHIP for the case when $I$ becomes an infinite
index set although it was used to define CHIP for finite convex sets
in a Hilbert space.

\begin{pro}
Let $\{A_i: i\in I\}$ be a collection of convex sets in $X$ with the
nonempty intersection. If $\{A_i: i\in I\}$ has the strong CHIP,
then $\{A_i: i\in I\}$ has the weak normal CHIP.
\end{pro}
\begin{proof}
Denote $A:=\bigcap_{i\in I}A_i$ and let $x\in \bigcap_{i\in I}A_i$.
We prove that $\{{\rm con}(A_i-x): i\in i\}$ has the weak normal
property. Since $\{A_i: i\in I\}$ has the strong CHIP at $x$, it
follows that $\{A_i: i\in I\}$ has the CHIP at $x$ and consequently
\eqref{5.3} holds. Thus, by Theorem 4.1, we only need to prove that
for any $x^*\in X^*$ there exists $\eta:=\eta({x^*})>0$ such that
\begin{equation}\label{5.4}
[0, x^*]\#\overline{\rm co}^{w^*}\Big(\bigcup\limits_{i\in I}({\rm
cone}(A_i-x))^{\circ}\Big)\subset \overline{\rm
co}^{w^*}\Big(\bigcup_{i\in I}\big(({\rm
cone}(A_i-x))^{\circ}\#{\eta}B_{X^*}\big)\Big).
\end{equation}
From the strong CHIP, one has
\begin{equation}\label{5.5}
N(A, x)=\sum_{i\in I}N(A_i, x).
\end{equation}
By Proposition 4.1(a) and \eqref{5.5}, we have that \eqref{5.4} is
equivalent to
\begin{equation}\label{5.6}
[0, x^*]\cap N(A, x)\subset \overline{\rm
co}^{w^*}\Big(\bigcup_{i\in I}\big(N(A_i,
x)\cap{\eta}B_{X^*}\big)\Big),
\end{equation}

Suppose that $x^*\in N(A, x)$. By \eqref{5.5}, there exist a finite subset
$I'\subset I$ and $x_i^*\in N(A_i, x)$ $(\forall i\in I')$ such that
$x^*=\sum_{i\in I'}x_i^*$. Take $\eta\in (0, +\infty)$ such that
$|I'|\cdot\sum_{i\in I'}\|x^*_i\|\leq \eta$ and let
$$
\mu_i:=\frac{1}{|I'|}\ \ {\rm and}\ \ \tilde{x}^*_i:=|I'|\cdot
x^*_i\ \ \forall i\in I'
$$
where $|I'|$ denotes the cardinality of the finite set $I'$. Then
$\sum_{i\in I'}{\mu_i}=1$, $\tilde{x}^*_{i}\in N(A_i, x)\cap \eta
B_{X^*}$ and $x^*=\sum_{i\in I'}\mu_i\tilde{x}^*_{i}\in{\rm
co}\big(\bigcup_{i\in I}(N(A_i, x)\cap{\eta}B_{X^*})\big)$. This
means that \eqref{5.6} holds with $\eta>0$. If $x^*\not\in N(A, x)$,
\eqref{5.6} holds with any $\eta>0$. Therefore, $\{{\rm cone}(A_i-x):
i\in I\}$ has the weak normal property and consequently the weak
normal CHIP for $\{A_i: i\in I\}$ holds. The proof is completed.
\end{proof}

Based on Theorem 4.3, we have the following theorem on the
characterization of normal CHIP in terms of the extended Jamenson
property.

\begin{them}
Let $\{A_i: i\in I\}$ be a collection of convex sets in $X$ with the
nonempty intersection. Suppose that $\{A_i: i\in I\}$ has the CHIP.
Then the following statements are equivalent:

$\rm (i)$ $\{A_i: i\in I\}$ has the normal CHIP;

$\rm (ii)$ for any $x\in\bigcap_{i\in I}A_i$ there exists $\eta>0$
such that $\{N(A_i, x): i\in I\}$ has property $(G_{\eta})$;

$\rm (iii)$ $\{T(A_i, x): i\in I\}$ has the normal property for each
$x\in\bigcap_{i\in I}A_i$ with the same constant.
\end{them}
\begin{proof}
Let $x\in \bigcap_{i\in I}A_i$. By the CHIP and Remark 5.1(b), one
has \eqref{5.3} holds, i.e., $\{{\rm cone}(A_i-x): i\in I\}$ has the
closed intersection property. Then Theorem 4.3 implies that $\{{\rm
cone}(A_i-x): i\in I\}$ has the normal property if and only if there
exists $\eta>0$ such that $\{({\rm cone}(A_i-x))^{\circ}: i\in I\}$
has property $(G_{\eta})$. Hence (i)$\Leftrightarrow$(ii) holds
since $({\rm cone}(A_i-x))^{\circ}=(A_i-x)^{\ominus}=N(A_i, x)$.
Noting that each $T(A_i, x)$ is a closed convex cone and $N(A_i,
x)=(T(A_i, x))^{\circ}$, it follows from Theorem 4.3 that (ii) is
equivalent to (iii). The proof is completed.
\end{proof}

In the next subsection, we mainly study linear regularity of
arbitrarily many convex sets in a Banach space. The notion of linear
regularity was used by Bauschke and Borwein [5] to establish a linear convergence rate of iterates generated by the
cyclic projection algorithm for finding the projection from a point
to an intersection of finitely many closed convex sets. This concept
has been studied and applied in optimization and mathematical programming (cf. [16,
17, 19, 27, 28, 29, 30] and references therein). The main purpose of this
subsection is to provide the equivalent interrelationship among linear
regularity, the normal CHIP and the normal property. We begin with
the concept of the linear regularity for arbitrarily many convex
sets in a Banach space.\\

\noindent{\bf Definition 5.2.} Let $\{A_i: i\in I\}$ be a collection of
convex sets in $X$ such that $A:=\bigcap_{i\in I}A_i$ is nonempty.

(i) $\{A_i: i\in I\}$ is said to have the {\it linear regularity
property}, if there exists $\gamma>0$ such that
\begin{equation}\label{5.7}
d(x, A)\leq\gamma\sup_{i\in I}d(x, A_i)\ \ \forall x\in X.
\end{equation}

(ii) $\{A_i: i\in I\}$ is said to have the {\it bounded linear
regularity property}, if for any $\rho>0$ there exists
$\gamma_{\rho}>0$ such that
\begin{equation}\label{5.8}
d(x, A)\leq\gamma_{\rho}\sup_{i\in I}d(x, A_i)\ \ \forall x\in \rho
B_X.
\end{equation}

The following theorem is the equivalence between linear regularity
and the uniform normal property.

\begin{them}
$\{A_i: i\in I\}$ be a collection of convex sets in $X$ such that
$A:=\bigcap_{i\in I}A_i$ is nonempty. Then $\{A_i: i\in I\}$ has
the  linear regularity property if and only if it has the uniform
normal property.
\end{them}
\begin{proof}
{\bf The sufficiency part}. By the uniform normal property, there
exists $\eta>0$ such that
\begin{equation}\label{5.9}
\bigcap_{i\in I}(A_i+\delta\eta B_X)\subset A+\delta B_X\ \ \forall
\delta>0.
\end{equation}
Let $x\in X\backslash A$. Then $\delta_0:=\sup_{i\in I}d(x, A_i)\in
(0, +\infty)$ and consequently
\begin{equation}\label{5.10}
x\in \overline{A_i+\delta_0B_X}\ \ \forall i\in I.
\end{equation}
Take arbitrary $\varepsilon>0$. By \eqref{5.9} and \eqref{5.10}, we
obtain
$$
x\in \bigcap_{i\in I}(\overline{A_i+\delta_0B_X})\subset
\bigcap_{i\in I}\big(A_i+(1+\varepsilon)\delta_0B_X\big)\subset
A+\frac{1+\varepsilon}{\eta}\delta_0 B_X.
$$
This implies that
$$
d(x,
A)\leq\frac{1+\varepsilon}{\eta}\delta_0=\frac{1+\varepsilon}{\eta}\sup_{i\in
I}d(x, A_i).
$$
Thus the linear regularity for $\{A_i: i\in I\}$ holds with
$\gamma:=\frac{1+\varepsilon}{\eta}$.

{\bf The necessity part}. By the linear regularity, there exists
$\gamma>0$ such that \eqref{5.7} holds. Take arbitrary $\eta:=(0,
1/\gamma)$ and let $\delta>0$. Then for any $x\in \bigcap_{i\in
I}(A_i+\delta \eta B_X)$, one has $\sup_{i\in I}d(x, A_i)\leq
\delta\eta$ and it follows from \eqref{5.7} that
$$
d(x, A)\leq\gamma\sup_{i\in I}d(x, A_i)\leq\gamma\eta\delta.
$$
From this, we get
$$
x\in \overline{A+ \gamma\eta\delta B_X}\subset A+\delta B_X
$$
where the inclusion follows as $\gamma\eta<1$. Hence $\{A_i: i\in
I\}$ has the uniform normal property with $\eta>0$. The proof is
completed.
\end{proof}

\noindent{\bf Remark 5.2.} We define
$$
\lambda_{UN}(A_i:i\in I):=\sup\{\eta\geq 0:
\eqref{5.9}\ {\rm holds}\ {\rm with} \ \eta\geq 0 \}
$$
and
$$
\gamma(A_i:i\in I):=\inf\{\gamma>0: \eqref{5.7}\ {\rm holds}\ {\rm
with} \ \gamma> 0  \}.
$$
Here we use the convention that the infimum over the empty set is
$+\infty$. Applying the proof of Theorem 5.2, one can verify that
\begin{equation}\label{5.11}
\gamma(A_i:i\in I)=\frac{1}{\lambda_{UN}(A_i:i\in I)}.
\end{equation}

The following corollary is immediate from Theorem 5.2 which is a
characterization of the bounded linear regularity property.

\begin{coro}
Let $\{A_i: i\in I\}$ be a collection of convex sets in $X$ such that
$A:=\bigcap_{i\in I}A_i$ is nonempty. Then $\{A_i: i\in I\}$ has
the bounded linear regularity property if and only if for any
$\rho>0$ there exists $\eta_{\rho}>0$ such that
\begin{equation}\label{5.12}
\rho B_X\cap\Big(\bigcap_{i\in
I}(A_i+\delta\eta_{\rho}B_X)\Big)\subset A+\delta B_X\ \ \forall
\delta>0.
\end{equation}
\end{coro}

\begin{them}
Let $\{K_i: i\in I\}$ be a collection of convex cones in $X$. Then the
following statements are equivalent:

$\rm (i)$ $\{K_i: i\in I\}$ has the  linear regularity property;

$\rm (ii)$ $\{K_i: i\in I\}$ has the normal property;

$\rm (iii)$ $\{K_i: i\in I\}$ has the closed intersection property
and there exists $\eta>0$ such that $\{K_i^{\circ}: i\in I\}$ has
property $(G_{\eta})$.
\end{them}

Applying Proposition 3.1(i), Theorems 4.3 and 5.2, one can
easily obtain the proof of Theorem 5.3.

The following theorem, as one main result in this paper, establishes
the equivalence among different concepts of the normal property, the
extended Jamenson property, the normal CHIP and the linear
regularity property in some sense.

\begin{them}
Let $\{A_i: i\in I\}$ be a collection of convex sets in $X$ with the
nonempty intersection. Suppose that $\{A_i: i\in I\}$ has the CHIP .
Then the following statements are equivalent:

$\rm (i)$ $\{A_i: i\in I\}$ has the normal CHIP;

$\rm (ii)$ $\{T(A_i, x): i\in I\}$ has the normal property for any
$x\in A:=\bigcap_{i\in I}A_i$;

$\rm (iii)$ for any $x\in\bigcap_{i\in I}A_i$ there exists $\eta>0$
such that $\{N(A_i, x): i\in I\}$ has property $(G_{\eta})$;

$\rm (iv)$ $\{T(A_i, x): i\in I\}$ has the linear regularity
property for any $x\in\bigcap_{i\in I}A_i$.
\end{them}
\begin{proof}
The equivalence (i)$\Leftrightarrow$(ii)$\Leftrightarrow$(iii)
follows from Theorem 5.1. It remains to prove the equivalence
between (i) and (iv).

Let $x\in\bigcap_{i\in I}A_i$. By the normal CHIP, we have
\begin{center}
$\{A_i: i\in I\}$ has the normal CHIP at $x$\ \ \ \ \ \ \ \ \\
 $\Longleftrightarrow$ $
\{ {\rm cone}(A_i-x): i\in I\}$ has the normal property.\\
\end{center}
From Theorem 5.3, the normal CHIP of $\{A_i: i\in I\}$ holds at $x$
if and only if there exists $\gamma>0$ such that
\begin{equation}\label{5.13}
d\Big(z, \bigcap_{i\in I}{\rm cone}(A_i-x)\Big)\leq \gamma\sup_{i\in
I}d\big(z, {\rm cone}(A_i-x)\big)\ \ \forall z\in X.
\end{equation}
Since $\{A_i: i\in I\}$ has the CHIP at $x$, by using Remark 5.1(b),
we have
\begin{equation}\label{5.14}
\overline{\bigcap_{i\in I}{\rm cone}(A_i-x)}=\bigcap_{i\in
I}\overline{\rm cone}(A_i-x)).
\end{equation}
Noting that $d\big(z,{\rm cone}(A_i-x)\big)=d\big(z,\overline{\rm
cone}(A_i-x)\big)=d\big(z, T(A_i, x)\big)$ and
\begin{eqnarray*}
&&d\Big(z, \bigcap_{i\in I}{\rm cone}(A_i-x)\Big)=d\Big(z,
\overline{\bigcap_{i\in I}{\rm cone}(A_i-x)}\Big)\\
&=&d\Big(z, \bigcap_{i\in I}\overline{{\rm
cone}}(A_i-x)\Big)=d\Big(z, \bigcap_{i\in I}T(A_i, x)\Big)
\end{eqnarray*}
(where the second equality holds by \eqref{5.14}) for all $z\in
X$, it follows from \eqref{5.13} that
$$
d\Big(z, \bigcap_{i\in I}T(A_i, x)\Big)\leq \gamma\sup_{i\in
I}d\big(z, T(A_i, x)\big)\ \ \forall z\in X.
$$
This implies that $\{T(A_i, x): i\in I\}$ has the linear regularity
property and thus (i)$\Leftrightarrow$(iv) holds. The proof is
completed.
\end{proof}

Recall that Li, Ng and Pong [19] studied the linear regularity for the
infinite system of arbitrarily many closed convex subsets in a
Banach space, and proved the following result on the equivalent
conditions for linear regularity characterized in terms of normal
cones. Readers can refer to [19, Theorem 4.5] for details and the
proof.

\begin{lem}
Let $\gamma>0$ and $\{A_i: i\in I\}$ be a collection of closed
convex sets in $X$ such that $A:=\bigcap_{i\in I}A_i\not=\emptyset$.
Suppose that $I$ is a compact metric space and $i\mapsto A_i$ is
lower
semicontinuous. Then the following statements are equivalent:\\
$\rm (i)$ For all $x\in X$, $d(x, A)\leq \gamma\sup_{i\in I}d(x,
A_i)$;\\
$\rm (ii)$ For all $x\in A$, $N(A, x)\cap
B_{X^*}\subset\overline{\rm co}^{w^*}\big(\bigcup_{i\in I}(N(A_i,
x)\cap \gamma B_{X^*})\big)$.
\end{lem}

By virtue of Lemma 5.1, we can use the normal property, extended
Jamenson property and CHIP to characterize the linear regularity
property. First we need to establish the following lemma.

\begin{lem}
Let $\gamma>0$ and $\{A_i: i\in I\}$ be a collection of closed
convex sets in $X$ such that $A:=\bigcap_{i\in I}A_i\not=\emptyset$.
Then (ii) in Lemma 5.1 holds if and only if $\{A_i: i\in I\}$ has
the CHIP and $\{N(A_i, x): i\in I\}$ has the property $(G_{\eta})$
for all $x\in A$ with the same constant $\eta:=1/\gamma$.
\end{lem}
\begin{proof}
{\bf The necessity part}. The CHIP is immediate from (ii) in Lemma
5.1. Let $x\in A$ and $x^*\in \overline{\sum\limits_{i\in I}N(A_i,
x)}^{w^*}\big\backslash \{0\}$. From the CHIP, we have
$\frac{x^*}{\|x^*\|}\in N(A, x)\cap B_{X^*}$. Then, there exist a
generalized sequence $\{x^*_k\}$ in ${\rm co}\big(\bigcup_{i\in
I}(N(A_i, x)\cap \gamma B_{X^*})\big)$ such that
$x^*_k\stackrel{w^*}\longrightarrow \frac{x^*}{\gamma\|x^*\|}$.
Thus, for each $k$, there are finite subset $I_k$ of $I$,
nonnegative scalars $\lambda_{i, k}$ and $\tilde{x}_{i, k}^*\in
N(A_i, x)\cap B_{X^*} (\forall i\in I_i)$ satisfying
\begin{equation*}
\sum_{i\in I_k}\lambda_{i, k}=1\ \ {\rm and}\ \ x^*_k=\sum_{i\in
I_k}\lambda_{i, k}\tilde{x}_{i, k}^*.
\end{equation*}
Let $ x^*_{i, k}:=\gamma\|x^*\|\lambda_{i, k}\tilde{x}^*_{i, k}\in
N(A_i, x) (\forall i\in I_k)$. By passing to the limit, we have
$$
\sum_{i\in I_k}x^*_{i, k}\stackrel{w^*}\longrightarrow x^*\ \ {\rm
and}\ \ \sum_{i\in I_k}\|x^*_{i, k}\|\leq \gamma\|x^*\|.
$$
This implies that $\{N(A_i, x): i\in I\}$ has property $(G_{\eta})$
with $\eta:=1/\gamma$.

{\bf The sufficiency part}. Let $x\in A$ and $x^*\in N(A, x)\cap
B_{X^*}$. By the CHIP property, one has $x^*\in
\overline{\sum\limits_{i\in I}N(A_i, x)}^{w^*}$. Noting that
$\{N(A_i, x): i\in I\}$ has the property $(G_{\eta})$, it follows
that there exist a generalized sequence $\{I_k\}$ of finite subsets
of $I$ and $x_{i, k}^*\in N(A_i, x)\backslash\{0\} (\forall i\in
I_k)$ such that
\begin{equation}\label{5.15}
\sum_{i\in I_k}x_{i, k}^*\stackrel{w^*}\longrightarrow x^*\ \ {\rm
and}\ \ \sum_{i\in I_k}\|x_{i, k}^*\|\leq\frac{1}{\eta}\|x^*\|\leq
\gamma.
\end{equation}
We denote
$$
\lambda_{i, k}:=\frac{\|x^*_{i, k}\|}{\sum_{j\in I_k}\|x^*_{j,
k}\|}\ \ {\rm and}\ \ \tilde{x}^*_{i, k}:=\frac{\sum_{j\in
I_k}\|x^*_{j, k}\|}{\|x^*_{i, k}\|}x^*_{i, k}\ \ \forall i\in I_k.
$$
From \eqref{5.15}, we have $\tilde{x}^*_{i, k}\in N(A_i, x)\cap
\gamma B_{X^*}$, $\sum\limits_{i\in I_k}\lambda_{i,
k}\tilde{x}^*_{i, k}\stackrel{w^*}\longrightarrow x^*$ and
$$
\sum\limits_{i\in I_k}\lambda_{i, k}\tilde{x}^*_{i, k}\in{\rm
co}\Big(\bigcup_{i\in I}(N(A_i, x)\cap \gamma B_{X^*})\Big).
$$
Hence $x^*\in \overline{\rm co}^{w^*}\big(\bigcup_{i\in I}(N(A_i,
x)\cap \gamma B_{X^*})\big)$. The proof is completed.
\end{proof}

By applying Lemmas 5.1 and 5.2, the following theorem, as one main result in
this paper, is immediate from Theorems 5.4 and 4.3.
\begin{them}
Let $\{A_i: i\in I\}$ be a collection of closed convex sets in $X$
with the nonempty intersection. Suppose that $I$ is a compact metric
space and that $i\mapsto A_i$ is lower semicontinuous. Then the following
statements are
equivalent:\\
$\rm (i)$ $\{A_i: i\in I\}$ has the linear regularity property;\\
$\rm (ii)$ $\{A_i: i\in I\}$ has the CHIP and there exists $\eta>0$
such that $\{T(A_i, x): i\in I\}$ has the normal property for all
$x\in \bigcap_{i\in I}A_i$ with the same constant $\eta$;\\
$\rm (iii)$ $\{A_i: i\in I\}$ has the CHIP and there exists $\eta>0$
such that  $\{N(A_i, x): i\in I\}$ has property $(G_{\eta})$ for all
$x\in \bigcap_{i\in I}A_i$\\
$\rm (iv)$ $\{A_i: i\in I\}$ has the CHIP and there exists $\gamma>0$
such that $\{T(A_i, x): i\in I\}$ has the linear regularity property for all
$x\in \bigcap_{i\in I}A_i$ with the same constant $\gamma$.
\end{them}

\noindent{\bf Remark 5.3.} Theorem 5.5 actually extends and improves [3, Theorem 6.2]. When $I$ is a finite index
set, Theorem 5.5 reduces to [28, Theorem 6.1].


\end{document}